\documentclass[11pt,leqno]{article}
\usepackage[colorlinks=false]{hyperref}

 \usepackage{xcolor}

\usepackage{amsmath,amsthm,amsfonts,amssymb,latexsym,amscd,enumerate}

\usepackage{palatino}
\usepackage[mathcal]{euler}

\usepackage{xy}
\xyoption{all}

\swapnumbers

\theoremstyle{plain}
\newtheorem*{theorem*}{Theorem} 
\newtheorem*{lemma*}{Lemma}
\newtheorem*{assumption*}{Assumption}

\newtheorem{theorem}{Theorem} 
\newtheorem{lemma}[theorem]{Lemma}

\newtheorem{proposition}[theorem]{Proposition}

\theoremstyle{definition}
\newtheorem{definition}[theorem]{Definition}
\newtheorem{example}[theorem]{Example}

\newtheorem{remark}[theorem]{Remark}

\theoremstyle{remark}

\numberwithin{theorem}{section}

\numberwithin{equation}{section}

\newcommand{\R}{\mathbb{R}}
 
\newcommand{\N}{\mathbb{N}}
\newcommand{\Z}{\mathbb{Z}} 
\newcommand{\T}{\mathbb T}

\renewcommand{\S}{\mathcal{S}}
\newcommand{\D}{\mathcal{D}}

\DeclareMathOperator{\order}{order}

\newcommand{\defN}{\mathbb N}

\newcommand{\G}{{\mathscr{G}}}
\renewcommand{\H}{{\mathscr{H}}}

\begin{document} 

\title{Euler-Like Vector Fields, Deformation Spaces and Manifolds with Filtered Structure} 

\author{Ahmad Reza Haj Saeedi Sadegh and Nigel Higson}

\date{}

\maketitle

\begin{abstract}
\noindent Let $M$ be a smooth submanifold of a smooth manifold $V$.  Bursztyn, Lima and Meinrenken   defined the concept of an Euler-like vector field  associated to the embedding of $M$ into $V$, and proved that  there is a bijection between germs of Euler-like   vector fields and germs of tubular neighborhoods of $M$. We shall present a new view of this result by characterizing Euler-like vector fields algebraically and examining their relation to the  deformation to the normal cone from algebraic geometry.  Then we shall extend our algebraic point  of view to   smooth manifolds that are equipped with Lie filtrations, and define deformations to the normal cone and  Euler-like vector fields in that context.     Our algebraic construction of the deformation to the normal cone  gives a new approach to  Connes' tangent groupoid and its generalizations to  filtered manifolds.  In addition, Euler-like vector fields   give rise to preferred coordinate systems on filtered manifolds.
\end{abstract}

\section{Introduction}

The purpose of this paper is to examine from an algebraic point of view the concepts of  Euler-like vector field  and   deformation  to the normal cone in the theory of smooth manifolds. We shall  relate the two, and then extend them  from smooth manifolds  to so-called filtered manifolds.  As an application, we shall obtain new views of Connes' tangent groupoid  and its generalizations.

 Recall that the  \emph{Euler vector field} on  a finite-dimensional real vector space $V$ is the infinitesimal generator  of   scalar multiplication. Thus if $f$ is a smooth function on $V$, then its derivative in the direction of the Euler vector field is 
\begin{equation}
\label{eq-euler}
E(f)(v) = \frac{d}{dt}\Big \vert _{t=0} f(e^tv).
\end{equation}
If  $
x_1,\dots ,x_n 
$
is any linear coordinate system on $V$ (or in other words, a basis for the dual vector space $V^*$),  
then 
\begin{equation}
\label{eq-euler-coordinates}
E = x_1 \frac{\partial\phantom{x}}{\partial x_1}  + \cdots + x_n  \frac{\partial\phantom{x}}{\partial x_n}  .
\end{equation}
The Euler vector field is also characterized by the property that  if $f$ is a smooth homogeneous    polynomial  on $V$   of degree $q$, then
\begin{equation}
\label{eq-euler-polynomials}
	E(f) = q\cdot f .
\end{equation}
The concept of Euler vector field extends easily to vector bundles: if $V$ is the total space of a smooth, real vector bundle over a smooth manifold $M$, then the \emph{Euler vector field} on $V$ is given by the  formula \eqref{eq-euler} above, or equivalently by the obvious variations of \eqref{eq-euler-coordinates} or \eqref{eq-euler-polynomials}.  On each fiber, the Euler vector field of the bundle restricts to the Euler vector field of the fiber.

We shall be concerned  with an extension of the concept of Euler-vector field to manifolds.   Recall first  that a smooth function $f$ on $V$ \emph{vanishes to order $q\ge 1$} on  a submanifold $M$ if $Df$ vanishes on $M$ for every linear differential operator $D$ on $V$ of order $q{-}1$ or less.

\begin{definition}[See {\cite[Definition 2.5]{BLM}}\footnote{In \cite{BLM} it is required that Euler-like vector fields be complete.  That is not necessary for our purposes, and does not affect the results below, which concern germs of Euler-like vector fields near $M$.  See also Remark~2.8 in \cite{BLM}.}]
If $M$ is a smooth embedded submanifold of a smooth manifold V, then an \emph{Euler-like vector field} for the embedding of $M$ into $V$ is a vector field $E$ on $V$ with the property that if $f$ is a smooth function on $V$ that vanishes on $M$ to order $q\ge 1$, then
\[
	E(f) = q\cdot f + r,
	\]
where the remainder $r$ is a smooth function that vanishes to order $q{+}1$ or higher.
\end{definition}

\begin{remark}
Actually the above condition for $q=1$ implies the conditions for all $q$.  But for later purposes  it will be convenient to phrase the definition as we did.
\end{remark}

If $V$ is the total space of a vector bundle over $M$, then the Euler vector field on $V$ is   Euler-like for the embedding of $M$ into $V$ as the zero section.  More generally, recall that a \emph{tubular neighborhood} of $M$ in $V$ is a diffeomorphism from an open neighborhood of the zero section in the total space of the normal bundle 
\[
N_VM = TV\big \vert _M \Big / TM
\]
to an open neighborhood of $M$ in $V$ such that: 
\begin{enumerate}[\rm ({1}.1)]
\item the diffeomorphism   is the identity on $M$ (where $M$ is embedded in the normal bundle as the zero section), and
\item 
\label{tubular-nbd-cond}
the differential of the diffeomorphism, restricted to vertical tangent vectors, induces the identity map from   $N_VM$ to itself.
\end{enumerate}
If $E$ is the Euler vector field on the normal bundle, then any tubular neighborhood embedding carries $E$ to an Euler-like vector field defined in a neighborhood of $M$ in $V$.  Let us call this  the Euler-like vector field \emph{associated} to the tubular neighborhood embedding.

 Bursztyn, Lima and Meinrenken proved  the following attractive result:

\begin{theorem}[See {\cite[Proposition 2.6]{BLM}}]
\label{thm-blm}
 The correspondence that assoc\-iates to a tubular neighborhood embedding its associated Euler-like vector field determines a \emph{bijection} from germs of tubular neighborhoods to germs of Euler vector fields.
\end{theorem}

  The theorem has a number of applications, and reader is referred to  \cite{BLM} for full details, but here is a simple   example.  Let $(M,\omega)$ be a symplectic manifold and let $m$ be a point in $M$.  By the Poincar\'e lemma there is a $1$-form $\alpha$ on $M$ such that $d\alpha = \omega$ near $m$. In addition, $\alpha$ can be chosen so that the vector field $X$ defined by $\iota_X \omega = 2\alpha$ is Euler-like for the embedding of $\{m\}$ into $M$ (note that this is a first-order condition on the coefficients of $\alpha$ at $m$; simple linear algebra shows it can be satisfied).  The corresponding tubular neighborhood identifies $\omega$ with a $2$-form on $T_mM$ having constant coefficients in any linear coordinate system.  This proves the Darboux theorem.
  
Other applications stem from the fact that any affine combination, with $C^\infty$-function coefficients, of Euler-like vector fields is again an Euler-like vector field.  So for example equivariant tubular neighborhood embeddings for compact group actions can  be constructed  by averaging Euler-like vector fields. In addition, equivariant forms of the Darboux theorem, and more generally equivariant normal form theorems, can be proved this way.

We shall examine Theorem~\ref{thm-blm} from the perspective of the \emph{deformation to the normal cone} that is associated to the embedding of $M$ into $V$, which in this paper we shall simply call  the \emph{deformation space} associated to the embedding.  
Among other things, the deformation space $\N_V M$  is a smooth manifold  that is equipped with a submersion onto $\R$. The fibers of this submersion over all nonzero $x \in \R$ are copies of $V$, while the fiber over $x=0$ is the normal bundle for the embedding of $M$ into $V$.
So the deformation space   may be described, as a set, as a disjoint union
\begin{equation*}
\N_VM = N_VM {\times} \{ 0\}\,\, \sqcup\, \bigsqcup _{\lambda \in \R ^{\times}}  V {\times} \{ \lambda\}.
\end{equation*}
  See Section~\ref{sec-deformation} for further details, including, most importantly, a review of the smooth manifold structure on $\N_VM$.

Here is a sketch of the proof of Theorem~\ref{thm-blm}. If $E$ is an Euler-like vector field for $M\subseteq V$, then there is an associated  vector field $\boldsymbol{E}$ on $\N_VM$ that is vertical for the submersion to $\R$, restricts to a copy of $E$ on each fiber $V{\times}\{\lambda\}$, and restricts to the Euler vector field on the zero fiber $N_VM{\times}\{0\}$. See Lemma~\ref{lemma-char-euler-like} (this property characterizes Euler-like vector fields). 

There is also a canonical vector field $\boldsymbol{C}$ on the deformation space that restricts to $\lambda\cdot   \partial/\partial \lambda$ on the open set 
\[
V {\times} \R^{\times} = \N_V M \big\vert _{\R^{\times}}
\]
(moreover  $\boldsymbol{C}$ is vertical on the fiber over $0\in \R$, and is the negative of the Euler vector field there).  The formula 
\[
\lambda \cdot \boldsymbol{T} =   \boldsymbol{C} +\boldsymbol{E}  
\]
defines a third ``translation'' vector field $\boldsymbol{T}$ on the deformation space.  The  time $t$ flow map associated to the vector field $\boldsymbol{T}$ sends the  fiber of the deformation space over $\lambda {=}0$ to the fiber over $\lambda {=}t$, and its differential in the vertical direction is $t$ times the identity (compare condition (1.2) above). The $t{=}1$ map is defined in a neighborhood of the zero section in the normal bundle, and is a tubular neighborhood embedding.  This associates a tubular neighborhood embedding to the Euler vector field $E$, which is the main issue in proving Theorem~\ref{thm-blm}.

Our approach throughout will be algebraic, treating vector fields very explicitly as derivations of algebras of smooth functions, and so on.  In addition we shall follow the algebraic-geometric approach and   {define} the deformation space $\N_VM$ to be  the spectrum  of  the Rees algebra associated to the filtration of smooth functions on $V$ by order of vanishing on $M$.  This point of view   fits very well with our second purpose, which is   to study deformation spaces   in the context of \emph{filtered manifolds}.   

A filtered manifold is a smooth manifold   that is   equipped with an increasing filtration on its tangent bundle which is compatible with Lie brackets of vector fields; see Definition~\ref{def-filtered-manifold} for details.    This concept has arisen in a number of interrelated areas, including  partial differential equations  and  sub-Riemannian geometry.  More recently,  filtered manifolds  have received  attention in noncommutative geometry thanks to work in index theory by  Connes and Moscovici \cite{ConnesMoscovici95}, Ponge \cite{Ponge00,Ponge06} and Van Erp \cite{vanErp05,vanErp10a,vanErp10b}.  

Prominent in these noncommutative-geometric works are generalizations of   Connes'  \emph{tangent groupoid} \cite{Connes94}, which is the deformation space associated to the diagonal embedding of an ordinary smooth  manifold into its square.  Our algebraic point of view leads to a new perspective  on the tangent groupoid (new at least in index theory), and a new construction of its generalizations in filtered manifold theory.

A recurring theme in the theory of  filtered manifolds is the importance of a family of unipotent ``osculating groups'' parametrized by the points of a filtered manifold. These are   central to the theory of deformation spaces and the tangent groupoid, since the counterpart  in the filtered manifold context of the normal bundle from ordinary smooth manifold theory is a bundle of homogeneous spaces of unipotent groups. 
 One of our  main observations (which is very simple) is that  the osculating groups emerge     naturally from the algebraic approach that we are taking here; see Section~\ref{sec-filtered-manifolds}.  Once the unipotent groups have been understood, the construction of the tangent groupoid (or any other deformation space) for filtered manifolds  is a  near-verbatim copy of the    construction for ordinary manifolds.  

At the end of the paper we shall return to Euler-like vector fields, but now     in the context of filtered manifolds, and examine the preferred coordinate systems on filtered manifolds that they give rise to---see Remarks~\ref{remark-carnot1} and    \ref{remark-carnot2}.

We are grateful to Raphael Ponge for comments concerning the literature, and to Robert Yuncken and Erik van Erp for several enlightening conversations.  In addition we thank Eckhard Meinrenken, Ralf Meyer and the referee for suggesting many corrections to,  and improvements of, the original manuscript.

\section{Smooth Manifolds From Algebras}

In this section we shall give some elementary algebraic definitions that we shall use throughout the paper,  give   criteria   guaranteeing that the spectrum of an algebra carries a smooth manifold structure, and compare derivations on algebras to vector fields on  spectra in the manifold case.

\begin{definition} 
\label{def-spectrum}
Let $A$ be an associative and commutative\footnote{It is not necessary to assume commutativity, but the definitions that follow are not very interesting in the noncommutative case.}  algebra (with a multiplicative identity) over the field of real numbers. A \emph{character} of $A$ is a nonzero algebra homomorphism
\[
\varphi \colon A \longrightarrow \R.
\]  
 The \emph{spectrum} of $A$ is the set of all characters. We equip it with the topology of pointwise convergence, that is, the topology having the fewest open sets so that the evaluation maps 
\[
\widehat a \colon \varphi \longmapsto \varphi (a)
\]
 are continuous functions on the spectrum, for every $a\in A$.  \end{definition}

\begin{definition}
\label{def-sheaf}
  Let $A$ be as above. Denote by $\S_A$ the smallest subsheaf of the sheaf of continuous real-valued functions on the spectrum of $A$ that includes all global sections of the form
   \[
f =  g (\widehat{a_1}  , \dots, \widehat{a_k} ) ,
\]
where $k\in \mathbb N$, where $a_1,\dots, a_k\in A$, and where $g$ is a smooth, real-valued function    on $\R^k$.
\end{definition}

\begin{definition}
\label{def-smooth-generation}
Let $\S$ be a sheaf of real-valued   functions on a topological space $X$. Let  $\Omega\subseteq  X$ be an open subset.     We shall say that functions $h_1,\dots, h_n\in \S(\Omega)$ \emph{smoothly generate $\S(\Omega)$}  if for every $f\in \S(\Omega)$ there is a smooth function $g$ on $\R^n$ such that 
\[
f = g(h_1  , \dots, h_n).
\]
In the case where $X= \operatorname{Spectrum}(A)$ and $\S = \S_A$, we shall also say that elements $a_1,\dots, a_n\in A$ smoothly generate $\S_A(\Omega)$ if the functions $h_j = \widehat a_j \vert _\Omega$ satisfy the above condition.\end{definition}

\begin{lemma}
\label{lemma-mfld-characterization}
Let $A$ be a commutative algebra over the real numbers.  The spectrum of $A$ is a smooth manifold of dimension $n$, with $\S_A$ equal to its sheaf of smooth functions, if and only if  for every point    in the spectrum there is an open  neighborhood $\Omega$ of that point, and there exist $a_1,\dots, a_n$ in $A$, such that 
\begin{enumerate}[\rm (i)]

\item  the elements $a_1, \dots, a_n$ smoothly generate  $\S_A(\Omega)$, and 

\item the map 
\[
(\widehat a_1,\dots, \widehat a_n) \colon
\operatorname{Spectrum}(A) \longrightarrow \R ^n
\]
is a homeomorphism from   $\Omega$   to an open set in  $\R^n$.  

\end{enumerate}

\end{lemma}

\begin{proof}
If $a_1,\dots, a_n$ exist for every point in the spectrum, as in the statement of the lemma, then the spectrum is certainly a smooth $n$-manifold with local coordinates $\widehat a_1,\dots, \widehat a_n$.
Conversely, suppose that the spectrum is a smooth $n$-manifold in such a way that $\S_A$ is the sheaf of smooth functions.  
Let $\varphi$ be a point in the spectrum and let $x_1,\dots, x_n$ be local coordinates at $\varphi$.  By definition of $\S_A$, there are elements $a_1,\dots , a_N\in A$ and smooth functions $g_i$ on $\R^N$   for   $i = 1,\dots , n$ such that 
\[
x_i = g_i (\widehat a_1 ,\dots, \widehat a_N)
\]
near $\varphi$.  We can assume that $\widehat a_j (\varphi) = 0$, for all $j$. Since the $x_i$ are local coordinates, there are smooth functions $h_j$ on $\R^n$ for $j=1,\dots, N$ such that 
\[
\widehat a_j = h_j (x_1,\dots, x_n)
\]
near $\varphi$.    If $g\colon \R^N \to \R^n$ and $h\colon \R^n \to \R^N$ are the smooth functions whose components are $g_i$ and $h_j$, then $g\circ h = \operatorname{id}_{\R^n}$ near $0$. So $g$ is a submersion at $0$.  By  linear algebra and the inverse function theorem, there is an inclusion $k$ of $\R^n$ into $\R^N$ as a coordinate subspace so that the composition
\[
\R^n \stackrel k \longrightarrow \R^N \stackrel g \longrightarrow  \R^n
\]
is a local diffeomorphism at $0$.  If $k$ maps the $i$'th standard basis vector of $\R^n$  to the $k_i$'th standard basis vector of $\R^N$, then the elements $a_{k_1}, \dots, a_{k_n}$ have the properties (i) and (ii) in the statement of the lemma. 
\end{proof}

For the rest of this section we shall assume that the spectrum of $A$  {is} indeed  a smooth manifold, with $\S_A$ the sheaf of smooth functions.  

\begin{definition}
\label{def-smooth-X}
Let $X$ be a 
 derivation of the algebra  $A$,  and let $\Omega$ be an open subset of the spectrum of $A$. We shall say that $ X$ \emph{is compatible with a vector field $\widehat X$ on $\Omega$} if  the diagram 
  \[
  \xymatrix@C=50pt{
 A  \ar_{X}[d]   \ar^-{a\longmapsto \widehat a} [r]  & C^\infty (\Omega)\ar[d]^{\widehat X} \\
A \ar[r]_-{a \longmapsto \widehat a} &    C^\infty (\Omega) 
   }
   \]
   commutes. 
\end{definition}

An obvious necessary condition for $X$ to be compatible with a vector field on $\Omega$ is  that if $\Lambda\subseteq \Omega$ is any open subset, if 
\[
a,a_1,\dots, a_k\in A,
\]
 if $g$ is a smooth function of $k$ variables, and if 
\begin{equation}
\label{eq-X-and-smooth-calc0}
\widehat a\vert _\Lambda  =g (\widehat{a_1},\dots, \widehat{a_k})\vert _\Lambda,
 \end{equation}
then
\begin{equation}
\label{eq-X-and-smooth-calc}
\widehat {X(a)} \vert_\Lambda 
 = \sum_{i=1}^k \,\, \widehat{ X(a_i)}\vert _\Lambda  
 \cdot g_i\bigl (\widehat{a_1},\dots, \widehat{a_k}\bigr )  \vert _\Lambda 
 \end{equation}
where $g_i$ denotes the $i$'th partial derivative of $g$.  

\begin{definition}
\label{def-smooth-gen}
 If $\Omega$ is an open subset of the spectrum,  then we shall   say that a derivation $X$ of $A$ is \emph{smooth over $\Omega$} if \eqref{eq-X-and-smooth-calc} holds for  every open subset  $\Lambda\subseteq \Omega$, all  $a,a_1,\dots, a_k\in A$, and all $g$ as in \eqref{eq-X-and-smooth-calc0}.
\end{definition}

\begin{lemma}
\label{lemma-smooth-extension1}
If the spectrum of $A$ {is} a smooth manifold,  then every  derivation of $A$ that is smooth over an open subset $\Omega$ of the spectrum of $A$  is compatible with      a   unique vector field on $\Omega$. 
\end{lemma}

\begin{proof}
By Lemma~\ref{lemma-mfld-characterization}, around every point of the spectrum there is a neighborhood $\Lambda$, and elements $a_1,\dots, a_n$ of $A$, so that  $\widehat{a_1},\dots, \widehat {a_n}$ are  coordinate functions on $\Lambda$.  Since any vector field on $\Lambda$ is completely determined by its action on a system of   coordinate functions, we see that there is at most one vector field on $\Omega$ that is compatible with any given derivation $X$.  

As for existence, given local coordinates  of the type  $\widehat{a_1},\dots, \widehat {a_n}$ on some open subset $\Lambda$ of  $\Omega$,  we can define $\widehat X$ on $\Lambda$ by 
\begin{equation*}
\label{eq-X-and-smooth-calc1}
\widehat {X} \bigl (g  (\widehat{a_1},\dots, \widehat{a_n}  )\bigr )   = \sum_{i=1}^n  \widehat{ X(a_i)}  \cdot g_i\bigl (\widehat{a_1},\dots, \widehat{a_n}\bigr ) .
\end{equation*}
This is a vector field, it is compatible with  $X$ on $\Lambda$ by \eqref{eq-X-and-smooth-calc}, and it is independent of the choice of local coordinates, again by \eqref{eq-X-and-smooth-calc}.  So we obtain a   vector field defined on all of  $\Omega$, as required.
 \end{proof}

In our calculations it will be helpful to observe the following fact:

\begin{lemma} 
\label{lemma-smooth-extension2}
Let $\Omega$ be an open subset of the spectrum of $A$ whose complement has empty interior.  Every derivation of $A$  that is smooth over $\Omega$ is smooth over the full spectrum.
\end{lemma}

\begin{proof}
Let $\Lambda$ be an open subset of the spectrum.  By hypothesis, any identity of smooth functions \eqref{eq-X-and-smooth-calc0} over $\Lambda$ leads to an identity of the type \eqref{eq-X-and-smooth-calc} over $\Lambda\cap \Omega$.  But $\Lambda \cap \Omega$ is dense in $\Lambda$, so the identity \eqref{eq-X-and-smooth-calc} holds over  $\Lambda$.
\end{proof}

\section{The Deformation Space for Smooth Manifolds}
\label{sec-deformation}

Let  $V$ be a smooth manifold and let $M$ be a smooth, embedded submanifold   (both of them without boundary, as will always be the case in this paper).   In this section we shall review the construction of the \emph{deformation to the normal cone}, or \emph{deformation space}, associated to the inclusion of $M$ into $V$.  See \cite[Chapter 5]{Fulton} for the standard treatment in algebraic geometry and see for example \cite{Higson10}  for the $C^\infty$-version. 

We shall emphasize the algebraic aspects of the  construction.  These play only a modest role for ordinary manifolds,  but they will be helpful when  we consider filtered manifolds later on.   

Here is a summary of what we shall do. 
The \emph{deformation space} associated to the embedding of $M$ into $V$   may be described, as a set, as a disjoint union
\begin{equation}
\label{eq-def-space-descr1}
\N_VM = N_VM {\times} \{ 0\}\, \sqcup\, \bigsqcup _{\lambda \in \R ^{\times}}  V {\times} \{ \lambda\} ,
\end{equation}
as we noted in the introduction. It is given the weakest topology  so that the obvious maps to $\R$ and to $V$ are continuous, and so that, in addition, for every smooth function $a$ on $V$ that vanishes on $M$, the function 
\[
  \begin{aligned}
(X_m,0) &\longmapsto X_m(a) \\
(v,\lambda)  & \longmapsto \lambda^{-1} a(v)
\end{aligned} 
\]
is also continuous.  Here $X_m$ is a normal vector at $m\in M$, that is, a   vector in the quotient space $T_mV/T_mM$.  The value $X_m(a)$   is well defined because $a$ vanishes on $M$.  We shall prove that  the deformation space carries a smooth manifold structure  so that all the functions above are smooth (in fact they smoothly generate the sheaf of all smooth functions on the deformation space in the sense of Definition~\ref{def-smooth-generation}).

Now we proceed with the details.

\begin{definition}
\label{def-rees-alg}
Denote by  $A(V,M)$ the $\R$-algebra  of all Laurent polynomials
\[
f (t) = \sum_{q\in \Z} a_q  t^{-q}
\]
whose coefficients $a_q$  are smooth, real-valued  functions on $V$ that satisfy the condition
\[
q > 0 \quad \Rightarrow \quad \text{$a_q$ vanishes to order $q$ on $M$} 
\]
(we emphasize that by definition  only finitely $a_q$  are nonzero). The space $A(V,M)$  is indeed an algebra, because if $a_p$ vanishes to order $p$ on $M$, and $a_q$ vanishes to order $q$ on $M$, then the pointwise product $a_pa_q$ vanishes to order $p+q$.
The \emph{deformation space} $\defN_VM$ is the  {spectrum} of $A(V,M)$.
\end{definition}
 
 Our first objective is to identify $\N_VM$, defined as a spectrum, with   \eqref{eq-def-space-descr1}.  Associated  to $t\in A(V,M)$ is the continuous map 
\begin{equation}
\label{eq-R-submersion}
 \widehat t \colon \N_V M \longrightarrow \R 
 \end{equation}
 as in Definition~\ref{def-spectrum}, and we shall compute the fibers over each $\lambda\in \R$.  These are the spectra of the following algebras:
 
\begin{definition}
\label{def-A-sub-lambda}
For $\lambda \in \R$ denote by   $A_\lambda (V,M)$ the quotient of $A(V,M)$ by the ideal generated by $t{-}\lambda$.  
\end{definition}

\begin{lemma}
If $\lambda\in \R$ is   nonzero, then $A_\lambda (V,M)$  is isomorphic to $C^\infty(V)$ via evaluation of Laurent polynomials at $t=\lambda$. 
\end{lemma}

\begin{proof}
If the element $\sum a_q t^{-q}$ lies in the kernel of evaluation at $\lambda$, then 
\[
\sum a_qt^{-q} = (t-\lambda) \cdot \sum _{q}\bigl ( \sum_{j\ge 0} a_{q-j}\lambda^j\bigr ) t^{-q-1}, 
\]
and the right-most Laurent polynomial lies in $A(V,M)$, as required.
\end{proof}

To handle the case where $\lambda =0$ we need some notation.

\begin{definition}
For each integer $q >0$ denote by  $I_q(V,M)$ the ideal of smooth  functions on $V$  that vanish to order $q$ on $M$. Set $I_0(V,M) = C^\infty (V)$.
\end{definition}

The spaces $I_q(V,M)$ form a decreasing filtration of the algebra of smooth functions on $V$, and  we can form the associated graded algebra
\begin{equation}
\label{eq-assoc-graded}
  \bigoplus _{q \ge 0} I_{q}(V,M)/ I_{q+1}(V,M) .
\end{equation}
If $a\in I_q(V,M)$, then we shall write 
\begin{equation}
\label{eq-assoc-graded-notation}
\langle a \rangle_q\in I_q(V,M)/I_{q+1}(V,M)
\end{equation}
for the coset of $a \in I_q(V,M)$ in the degree $q$ component of \eqref{eq-assoc-graded}.

\begin{lemma}
\label{lemma-assoc-graded-alg-comp1}
The  algebra $A_0(V,M)$   is isomorphic to      the 
associated graded algebra 
\eqref{eq-assoc-graded}
   via the map 
\[
\pushQED{\qed} 
\sum_{q\in \Z} a_q t^{-q} \longmapsto \sum_{q\ge 0} \langle a_q\rangle_q .
\qedhere
\popQED
\]
\end{lemma}

It is now easy to compute the spectrum of $A_0(V,M)$. 
The degree zero part of $A_0(V,M)$ is  $C^\infty(M)$, and    each character of $A_0(V,M)$ restricts to evaluation at some point $m\in M$ on the degree zero part. The character  therefore  factors through the quotient algebra  $A_{0,m}(V,M)$ by the ideal in $A_0(V,M)$ generated by the vanishing ideal of $m$ in $C^\infty(M)$.

\begin{lemma}
\label{lemma-assoc-graded-alg-comp2}
There is a unique isomorphism from $A_{0,m}(V,M)$ to   
the algebra of   real-valued polynomial functions on the normal vector space 
$
  T_mV  / T_mM
$
for which  
 \[
\langle  a\rangle_1  \longmapsto \big [ X_m \mapsto   X_m (a) \bigr ].
 \]   
 for every normal vector $X_m$ and every smooth function $a$ on $V$ vanishing on $M$.  The spectrum of $A_{0,m}(V,M)$ identifies in this way with $
  T_mV  / T_mM
$. \qed
\end{lemma}

\begin{remark}
We shall prove a more general result in  Theorem~\ref{thm-filtered-A-zero-fiber}.
\end{remark}

Returning to the deformation space,   the above considerations  identify the fibers of \eqref{eq-R-submersion}  with $V$ when $\lambda {\ne} 0$, and with the normal bundle $N_VM$ when $\lambda{=}0$.  We obtain the description \eqref{eq-def-space-descr1}, as required.  As for the topology on $\N_VM$, since $A(V,M)$ is generated by:
\begin{enumerate}[\rm (a)]
\item  the element $t\in A(V,M)$, 
\item  the functions  $a\cdot t^0 \in A(V,M)$, where $a\in C^\infty (V)$, and 
\item   monomials  $a \cdot t^{-1}\in A(V,M)$, where $a$ vanishes on $M$.
\end{enumerate}
we find that the topology on $\N_VM$, viewed as a spectrum, agrees with the topology we described earlier.

\begin{theorem}
\label{thm-def-space-mfld}
The deformation space $\defN_VM$ is a smooth manifold.
\end{theorem}

\begin{proof}
We shall use Lemma~\ref{lemma-mfld-characterization}. The only nontrivial case is that of a character $\varphi$ in the fiber over $\lambda{=}0$, corresponding to a  normal vector $X_m$. Introduce smooth functions $x_1,\dots, x_n$  on  $V$ that are local coordinates in a neighborhood $U$ of  $m$ in $ V$, for which
\[
  M\cap U = \{ \, u\in U \,  : \, x_{k+1}(u) = \cdots = x_n(u) = 0 \, \}.
\] 
Now define $\Lambda\subseteq \N_VM$ to be the open set consisting of those  elements of the deformation space of the form $(u,\lambda)$ for $u\in U$ and $\lambda \ne 0$, or $(X_u,0)$ for $u\in M\cap U$. The elements
\begin{equation}
\label{eq-def-space-coords}
t, x_1 ,\, \dots\,,\, x_k ,\, x_{k+1}t^{-1},\, \dots \, ,\, x_{n}t^{-1} \in A(V,M)
\end{equation}
satisfy the hypotheses of Lemma~\ref{lemma-mfld-characterization}; if $W\subseteq \R^n$ is the image of $U$ under the coordinates $\{x_j\}$ on $V$, then  the homeomorphic  image of $\Lambda$ under the functions \eqref{eq-def-space-coords} is the open set  
\[
\bigl \{\, 
(\lambda, x_1,\dots, x_n) \, : \,  (x_1,\dots, x_k, \lambda x_{k+1}, \dots, \lambda x_n) \in W
\,\bigr  \} 
\]
in $\R^{n+1}$; and  the smooth generation statement in the lemma follows from the Taylor expansion for smooth functions on $V$.
\end{proof}

\section{The Tangent Groupoid}
\label{sec-tangent-groupoid}

 In this section we shall briefly describe the special features of the   deformation space associated to  the diagonal embedding of a smooth manifold into its square. This is in preparation for Section~\ref{sec-tangent-groupoid-filtered} where a more complicated version of the same thing will be considered.

\begin{definition}
\label{def-tangent-groupoid}
Let $M$ be a smooth manifold.  The \emph{tangent groupoid} $\T M$ is the deformation space associated to the diagonal embedding of $M$ into $M{\times} M$.
\end{definition}

The name \emph{tangent groupoid}  is due to  Connes, who explained the importance  of the tangent groupoid in index theory.  See \cite[Chapter 2, Section 5]{Connes94}, and see   \cite{CarrilloRouse08} for more details concerning  the construction of the tangent groupoid  using smooth manifold techniques. 

As the name promises, $\T M$ is not only a smooth manifold but a Lie groupoid (see \cite[Chapter 5]{MoerdijkMrcun03} for background information on Lie group\-oids).  The source, target and other structure maps are all obtained from the following functoriality property of the deformation space construction: from a commutative diagram of smooth manifolds and submanifolds 
\[
\xymatrix@C=40pt{
M_1 \ar[r] \ar[d] & M_2 \ar[d]\\
V_1 \ar[r] & V_2
}
\] 
(where the horizontal arrows are any smooth maps) we obtain a smooth map $\N _{V_1} M_1 \to \N_{V_2} M_2$. Moreover if the horizontal arrows are  submersions, then so is the map of deformation spaces.

In the case at hand, think of $M$ as diagonally embedded in   $M{\times}M$ and $M {\times} M {\times} M$, and note that the deformation space for  the identity embedding of $M$ in itself is simply $M{\times} \R$. The first and second coordinate projections 
\[
\xymatrix@C=40pt{
M \ar@{=}[r]  \ar[d]    \ar[d] & M\ar[d] \\
M {\times } M \ar@<-.3ex>[r]_-{\pi_2} \ar@<.3ex>[r]^-{\pi_1}& M 
}
\]
determine  target and source maps 
\[
\xymatrix@C=40pt{
\T M  \ar@<-.3ex>[r]_-{t} \ar@<.3ex>[r]^-{s} & M {\times}\R .
}
\]
The unit map is determined by the diagonal inclusion of $M$ into $M{\times}M$, and the  inverse map is determined by the flip map on $M{\times}M$. Finally the space of composable elements in $\T M$,  
\[
\T M^{(2)}= \bigl \{\,  (\gamma_1,\gamma_2 ) \in \T M \times   \T M :s(\gamma_1) = t(\gamma_2) \, \bigr\}
\]
is the deformation space for the diagonal embedding of $M$ into $M{\times}M {\times} M$, while  the projection 
\[
\xymatrix{ 
M{\times}M {\times} M \ar[r] & M{\times} M
}
\]
onto the first and third factors gives the composition law for   $\T M$.  

All these maps are easy to compute explicitly in terms of the description \eqref{eq-def-space-descr1} of the deformation space.  The part of  $\T M$ over  each $\lambda \in \R$ is a subgroupoid, and when $\lambda \ne 0$ we obviously obtain a copy of the pair groupoid of $M$.  When $\lambda =0$ we obtain the tangent bundle $TM$, viewed as a bundle of abelian Lie groups over $M$; this   computation will be carried out in a more general context in Section~\ref{sec-tangent-groupoid-filtered}.

\section{Vector Fields on the Deformation Space}
\label{sec-vector-fields-deformation-space}

In this section we shall give a proof of Theorem~\ref{thm-blm} (the theorem of Bursztyn, Lima and Meinrenken) using vector fields on the deformation space.     First we shall prove the \emph{existence} of compatible  tubular neighborhood embeddings:   

\begin{theorem}
\label{thm-tubular-constr}
If $E$ is an Euler-like vector field for the inclusion of $M$ into $V$, then there is a  tubular neighborhood diffeomorphism
\[
\Phi \colon  N_VM \longrightarrow V
\]
\textup{(}defined on a neighborhood of the zero section\textup{)}  that carries the Euler vector field on the normal bundle to the germ of $E$ near $M$.
\end{theorem}

The first step in our proof is to  construct from $E$  a vector field on the deformation space.  To start, let us denote by $\boldsymbol{E}$  the vector field on 
 $V{\times}\R^{\times}$  that is tangent to the fibers of the projection map $V{\times}\R^{\times}\to \R^{\times}$, and that is a copy of $E$ on each fiber $V{\times}\{\lambda\}$.
 
\begin{lemma}
\label{lemma-char-euler-like}
If $E$ is Euler-like, then the  vector field $\boldsymbol{E}$ above extends uniquely to a vector field on $\N_VM$. The extension is   tangent to the fibers of the projection $\N _V M \to \R$, and the restriction to the fiber over $0\in \R$ is the Euler vector field on the normal vector bundle  $N_VM$.
\end{lemma}

\begin{proof}
The extension is unique, if it exists, because $V {\times} \R^{\times}$ is dense in $\N _V M$.
The existence of the extension, and its properties, are easy to check in local coordinates.  Alternatively,
if $E$ is Euler-like, then, since $E$ preserves the order of vanishing of functions on $M$, the formula 
\begin{equation*}
  \sum a_q t^{-q} \longmapsto  \sum  E (a_q) t^{-q} 
\end{equation*}
defines a derivation of $A(V,M)$ that is compatible, in the sense of Definition~\ref{def-smooth-X}, with the vector field $\boldsymbol{E}$ on $V{\times}\R^{\times}$.   We therefore obtain a smooth extension of  $\boldsymbol{E}$  from Lemmas~\ref{lemma-smooth-extension1} and~\ref{lemma-smooth-extension2}.  It follows from the definition of an Euler-like vector field that the restriction  of this smooth extension to $N_VM$ is the Euler vector field.
\end{proof}

\begin{remark} 
The lemma actually \emph{characterizes} Euler-like vector fields:  if  $X$ is a vector field on $V$, and if  the extension to a vector field $\boldsymbol{X}$ to $V{\times}\R^{\times}$, as above, further extends  a vector field  on $\N_VM$ that then restricts to the Euler vector field on the vector bundle $N_VM$, then $X$ is   Euler-like.
\end{remark}

Next we shall introduce a canonical vector field on the deformation space.

\begin{lemma}
The formula 
\[
\gamma _s \colon \left \{ 
		\begin{aligned}  (v,\lambda) & \longmapsto (v, e^s\lambda) \\
					(X,0) & \longmapsto (e^{-s}X,0)\end{aligned}\right .
\]
defines a smooth action of the Lie group $\R$ on the deformation space $\N _VM$.
\end{lemma}
 
 \begin{proof}
 This is again easy to check directly in the local coordinates of Theorem~\ref{thm-def-space-mfld}. From the algebraic point of view, it suffices to note  that the geometric flow is associated to the morphism
  \[
 \gamma \colon A(V,M) \longrightarrow A(V,M)\otimes _{\R} C^\infty(\R) 
 \]
 defined by the formula
 \[
 \gamma \colon \sum a_{q}t^{-q} \longmapsto  \sum a_q t^{-q}\otimes e^{-tq}
 \]
 (the tensor product here is the ordinary algebraic tensor product).
 \end{proof}

\begin{definition}
We shall denote by $\boldsymbol{C}$ the vector field on $\N_VM$ that generates the flow   $\{ \gamma_s\}$ above.   Note that 
$\boldsymbol{C}$ restricts to the vector field 
$
 \lambda  \cdot {\partial} /{\partial \lambda} ,
$
on $V {\times} \R^{\times}$, while on the zero fiber of $\N_VM$  it agrees with the negative of the Euler vector field on the normal bundle.\end{definition}

Now we shall combine $\boldsymbol{E}$ with $\boldsymbol{C}$ to obtain a new vector field  $\boldsymbol{T}$ on $\N_VM$.

\begin{lemma}
\label{lemma-T-field}
 Let $E$ be an Euler-like vector field for the inclusion of $M$ into $V$, and let $\boldsymbol{E}$ be the associated vector field on $\N_VM$, as in Lemma~\textup{\ref{lemma-char-euler-like}}.  The vector field
\[
\boldsymbol{T} =  {\lambda}^{-1}   \boldsymbol{E} + \frac{\partial} {\partial \lambda}
\]
on the open subset $V {\times} \R^{\times}\subseteq \N_V M$ extends to a \textup{(}smooth\textup{)} vector field on $\N_VM$ with 
\[
 \lambda \cdot \boldsymbol{T}   = \boldsymbol{C}  + \boldsymbol{E} .
\]
\end{lemma}

\begin{proof} 
If $E$ is Euler-like, then the formula 
\begin{equation*}
  \sum a_q t^{-q} \longmapsto  \sum (E (a_q) - q a_q)t^{-(q+1)} ,
\end{equation*}
defines a derivation of $A(V,M)$.  The derivation is compatible with the vector field $\boldsymbol{T}$ over the open set  $V {\times} \R^{\times}\subseteq \N_V M$, and since the complement of this open set has  empty interior it follows from Lemma~\ref{lemma-smooth-extension2} that  the derivation is compatible (in the sense of Definition~\ref{def-smooth-X}) with a unique vector field on all of $\N_VM$.
\end{proof}
 
 It is clear from  its definition that the vector field $\boldsymbol{T}$ on $\N_VM$ is $\widehat t$-related to the vector field $d/d\lambda$  on $\R$ (recall from \eqref{eq-R-submersion}   that $\widehat{t}$ is the natural projection from $\N _VM$ to $\R$). As a result,  the time  $t{=}1$ flow map for the vector field $\boldsymbol{T}$  maps the $\lambda {= }0$ fiber $N_VM\subseteq \N_VM$  to the $\lambda{=}1$ fiber $M\subseteq V$ (although we   need to be a bit careful about the domain of definition of the flow map).  We shall show  that  this fiber mapping  is a tubular neighborhood, and that  it carries the Euler vector field on the normal bundle to the Euler-like vector field $E$.

\begin{definition}
\label{def-of-tau-and-phi}
Denote by  $\{ \tau_s\}$ the local flow on $\N_V M$ associated to the vector field $\boldsymbol{T}$  in Lemma~\ref{lemma-T-field}.  
\end{definition}

Recall that the maps $\tau_s$  assemble into a smooth map
\[
\tau \colon \R \times \N_V M \longrightarrow \N_V M
\]
that is defined   on some   neighborhood of $\{0\}{\times} \N_V M$ in $\R{\times} \N_V M$, such that 
\begin{equation*}
\boldsymbol{T}  (f)(w) = \frac{d}{ds}\Big \vert _{s=0} f (\tau_s(w))
\end{equation*}
for all smooth functions $f$ on $ \N_V M$ and all $w\in  \N_V M$, and 
\begin{equation*}
\tau_{s+t}(w) = \tau_s(\tau_t(w))
\end{equation*}
in a neighborhood of $\{ 0\} {\times} \{0\} {\times} \N_V M$ in $\R{\times} \R {\times} \N_V M$. In all these formulas, we are writing $\tau_s(w) = \tau(s,w)$.   For $s\ne 0$, then we shall write the restriction of the flow $\tau_s$ to the fiber of $N_VM$ over $\lambda{=}0$ in the form
\[
\N_VM  \ni (X,0)  \stackrel {\tau_s} \longmapsto (\varphi_s(X), s) \in \N_V M. 
 \]
For any open subset $U\subseteq N_VM$ with compact closure, and all sufficiently small $|s|$, the map $\varphi_s$ is a diffeomorphism
from $U$ to an open subset of $V$.

  \begin{lemma}
  \label{lem-varphi-s-derivative}
Let $f$ be a smooth function on $V$ that vanishes on $M$. There is a smooth function $h\colon V \to \R$ that vanishes to order $2$ such that 
 \[
  \frac{d}{ds}     f (\varphi_s (X_m))   =  s^{-1} f(\varphi_s(X_m)) + s^{-1}h(\varphi_s(X_m)) 
 \]
 for every $X_m\in N_VM$ and all sufficiently small $|s|$.
 \end{lemma}
 
 \begin{proof}
  Since $E$ is an Euler vector field, we can write 
 \[
 E(f) = f +h,
 \]
 where $h$ vanishes on $M$ to order $2$.
 Now define $\boldsymbol{f}$ to be the composition 
 \[
 \N_V M\longrightarrow V\times \R\longrightarrow V\stackrel f \longrightarrow \R .
 \]
 Then by definition of $\varphi_s$ and the flow $\tau_s$, 
 \[
   \frac{d}{ds}     f (\varphi_s (X_m))   =   \frac{d}{ds}     \boldsymbol{f} (\tau_s (X_m,0))  = \boldsymbol{T}_{\tau_s(X_m,0)}(\boldsymbol{f})
   \]
  But it  follows from the definition  of $\boldsymbol{T}$ that 
  \begin{multline*}
   \boldsymbol{T}_{\tau_s(X_m,0)}(\boldsymbol{f}) = s^{-1} \boldsymbol{E}_{(\varphi_s(X_m),s)}(\boldsymbol{f})  \\
   = s^{-1}  {E} (f)(\varphi_s(X_m))  = s^{-1} f(\varphi_s(X_m)) + s^{-1} h(\varphi_s(X_m)),
   \end{multline*}
   as required.
 \end{proof}
 
The map $\varphi_s$ takes the zero section $M\subseteq N_VM$ identically  to $M\subseteq V$, because 
  $\boldsymbol{T} $ restricts to  $\partial/\partial \lambda$ on the submanifold
$ M {\times} \R \subseteq \N_V M $.
So for every $m\in M$ and all sufficiently small $|s|$  the derivative of $\varphi_s$  induces a map 
\begin{equation}
\label{eq-normal-derivative-of-phi}
\varphi_{s,*} \colon  T_m V / T_m M \longrightarrow T_m V / T_m M
\end{equation}

\begin{lemma}
\label{lem-normal-derivative-of-phi}
The  mapping \textup{\eqref{eq-normal-derivative-of-phi}} is $s \cdot \operatorname{id}$.
\end{lemma}
\begin{proof}
 We  shall   calculate the linear algebraic adjoint  $\varphi_s^*$ of the linear transformation \eqref{eq-normal-derivative-of-phi}. The   vector space of smooth functions on $V$ that vanish to first order on $M$ surjects onto the vector space dual of $T_m V / T_m M$ via the usual pairing of functions and tangent vectors, and functions that vanish to second order are in the kernel of the surjection.  Applying Lemma~\ref{lem-varphi-s-derivative}  we find that 
 \[
\frac{d}{ds} \varphi_s^* = s^{-1} \varphi_s^*  \colon  (T_m V / T_m M)^* \longrightarrow (T_m V / T_m M)^* ,
 \]
 and so by calculus
$ s^{-1} \varphi^*_s $ is a constant family of linear maps.  To evaluate the constant we shall compute the limit of $s^{-1} \varphi^*_s $ as $s\to 0$.
The function $s\mapsto \tau_s(X_m,0)$ is a smooth curve in $\N_VM$, with value $(X_m,0)$ at $s=0$, and the function $ (v,s)\mapsto s^{-1}f(v)$ is smooth on $\N_VM$, with values $(X_m,0)\mapsto X_m(f)$ when $s=0$.  So
 \[
 \lim_{s\to 0}  s^{-1} f(\varphi_s(X_m))  =  X_m(f) .
 \]
 As a result, if $[f]$ denotes the class in $ (T_m V / T_m M)^*$ determined by $f$, namely 
 \[
 [f]\colon X_m \longmapsto X_m(f)
 \]
 then 
 \[
 \varphi^*_s ( [f])= [f\circ \varphi_s]  \colon X_m \longmapsto s \cdot X_m (f)
 \]
 and so 
$
 \varphi_s ^* = s \cdot \operatorname{id} $, as required. 
 \end{proof}

Lemma~\ref{lem-normal-derivative-of-phi} tells us that for any $s$    the map $X_m \mapsto \varphi_s(s^{-1} X_m)$ is a tubular neighborhood mapping on the domain where it is defined, but this  may not be a neighborhood of the full zero section of $N_VM$. To remedy this problem, we shall use  the  Lie bracket relations among $\boldsymbol{E}$, $\boldsymbol{C}$ and $\boldsymbol{T}$, which  are as follows:
 \begin{equation}
 \label{eq-ECT-relations}
[\boldsymbol{T}, \boldsymbol{C} ] = \boldsymbol{T}
 , \quad 
[\boldsymbol{T}, \boldsymbol{E} ] = 0
, \quad \text{and} \quad 
[\boldsymbol{C}, \boldsymbol{E} ] = 0 
 \end{equation} 
 (note that it suffices to verify these relations on the dense set $V{\times}\R^{\times}\subseteq \N_VM$).

 \begin{lemma}
\label{lemma-integrated-bracket-relation}
 If $K$ is a compact subset of $ N_VM$ and $k>0$, then there exists $\varepsilon >0$ so that 
 \[
  \varphi_{e^ts}(X) = \varphi_s (e^{t}X)
 \]
 for all $X\in K$, all $|t|< k$, and all  $s\in (-\varepsilon, \varepsilon)$.
 \end{lemma}

\begin{proof}
It follows from the first relation in \eqref{eq-ECT-relations} that 
\begin{equation}
\label{eq-identity-generators2}
\tau_{e^ts} =  \gamma_t\circ \tau _s  \circ \gamma_{-t}  
 \end{equation}
 (to be precise, the identity is well-defined and correct on any given compact set $K$, and for $|t|$ bounded by any given $k$, as long as $|s|$ is sufficiently small). The formula in the lemma follows by evaluating both sides on $(X,0)$.
 \end{proof}

 \begin{proof}[Proof of Theorem~\ref{thm-tubular-constr}]
Choose a neighborhood of the zero section in $N_VM$ and a smooth positive function $s(m) $ so that $\varphi_s(X_m)$ is defined for all $X_m\in U$ and all $|s|< 2 s(m)$.
Using Lemma~\ref{lemma-integrated-bracket-relation}, we find that the germ of the map
\[
\Phi (X_m) = \varphi_{s(m)}(s(m)^{-1} X_m)
\]
near the zero section of $N_VM$ is independent of the map $m\mapsto s(m)$ and is a tubular neighborhood.  The second relation in \eqref{eq-ECT-relations} implies that $\Phi$ carries the Euler vector field on the normal bundle to $E$.
\end{proof}

 Theorem~\ref{thm-blm} also asserts that there is a \emph{unique} (germ of a)  tubular neighborhood embedding  that carries the Euler vector field to  any given Euler-like vector field.  We have nothing really new to say about this uniqueness statement, but for completeness here is a proof.
 
 \begin{lemma}
 \label{lem-commutes-with-scalars}
 Let $V$ be a finite-dimensional vector space and let $\Psi \colon U \to W$ be a diffeomorphism from one open neighborhood of $0\in V$ to another, with $\Psi(0) = 0$.  If $\Psi$ carries the Euler vector field to itself, near $0$, and if the derivative of $\Psi$ at $0$ is the identity, then $\Psi$ is the identity near $0$.
 \end{lemma}
 
 \begin{proof}
Let  $v$ be an element in   a ball around $0$ (with respect to some norm) that is contained in $U\subseteq V$.  
 Both of the curves $ \Psi(e^{-t} v )$ and $ e^{-t} \Psi (v)$ 
($t\ge 0$) have the same derivatives for all $t$, given by the negative of  the Euler vector field, and the same initial point at $t=0$.  Hence 
 \begin{equation}
 \label{eq-commutes-with-scalars}
 \Psi(e^{-t} v ) = e^{-t} \Psi (v) \qquad \forall t \ge 0.
 \end{equation}
 Now  by calculus, if $\Psi_*$ is the derivative of $\Psi$ at $0$, then there is a positive constant so that 
 \begin{equation}
 \label{eq-def-deriv}
\|  \Psi(u) -  \Psi_* u \| \le  \text{constant} \cdot \| u\|^2
 \end{equation}
for all $u\in U$ sufficiently close to $0$.   Writing $u= e^{-t}v$, multiplying \eqref{eq-def-deriv} by $e^t$, and using \eqref{eq-commutes-with-scalars},  we obtain 
 \[
  \| \Psi(v) - \Psi_* v \| \le   e^{-t } \cdot \text{constant} \cdot \| v\|^2 ,
  \]
  and so $\Psi (v) = \Psi_*v = v$.
 \end{proof}

\begin{proof}[Proof of the uniqueness statement in Theorem~\ref{thm-blm}]
If  two tubular neighborhood embeddings are given, under both of which  $E$ identifies with the Euler vector field, then the composition the first with the inverse of the second    is a diffeomorphism  $\Psi$   from one neighborhood of the zero section in the normal bundle $N_VM$ to another  that   fixes the zero section,  and
 carries the Euler vector field to itself.  By repeating the argument in Lemma~\ref{lem-commutes-with-scalars} we find that  if $X_m\in N_VM$ is contained in a ball  around $0$ that is contained in the domain of definition of $\Psi$, then 
 \[
 \Psi(e^{-t} X_m) = e^{-t} \Psi (X_m) \qquad \forall t \ge 0.
\]
Applying the projection $N_VM\to M$ to this equation and taking the limit as $t\to -\infty$, we find that $\Psi\colon N_VM \to N_V M$ is fiber-preserving near the zero section.  Now apply the previous lemma fiberwise, using the  condition (1.2) in the definition of tubular neighborhood embedding to verify that lemma's  derivative hypothesis.   \end{proof}

\section{Lie Filtrations and Unipotent Groups}
\label{sec-filtered-manifolds}

In this section we shall review the definition of a \emph{Lie filtration} on the tangent bundle of a smooth manifold,  due to Tanaka \cite{Tanaka70} (although the name for the concept that we use here was chosen by  Melin \cite{Melin82}) and give an algebraic description of  the unipotent \emph{osculating groups} that are attached to the points of a filtered manifold.

\begin{definition} 
\label{def-filtered-manifold}
Let $V$ be a smooth manifold.  A \emph{Lie filtration} on the tangent bundle $TV$ is an increasing  sequence of smooth vector subbundles 
 \[
 H^1    \subseteq H^2  \subseteq \cdots \subseteq H^r  = TV
\]
with the property that  if $X$ and $Y$ are vector fields on $V$, and also sections of $H^p$ and $H^q$, respectively,    then the Lie bracket $[X,Y]$ is a section of $H^{p+q}  $ (we set $H^{p+q}= TV$ if $p{+}q \ge r$).  An \emph{$r$-step filtered manifold} is a smooth manifold whose tangent bundle is equipped with a Lie filtration of length $r$, as above.
\end{definition}

\begin{remark}
The  concept of filtered manifold arises in a number of places.  Apart from \cite{Tanaka70} and \cite{Melin82}, see also    \cite{Morimoto93} and \cite{CapSlovak09}, for instance.  Some of the  treatments of filtered manifolds   in sub-Riemannian geometry  are particularly close to the perspective of this paper; see for example \cite[Secs. 4,5]{Bellaiche96} and \cite[Ch. 10]{Agrachevetal}.
 \end{remark}
  
We shall usually write $(V,H)$ to make explicit reference to the Lie filtration.   For simplicity we shall assume throughout  that the bundles $H^q$ in Definition~\ref{def-filtered-manifold} have constant rank, which of course they must have if $V$ is connected.  
 
\begin{example}
\label{ex-ordinary-filtered}
An ordinary smooth manifold is obviously a  $1$-step filtered manifold.  In the $1$-step case the constructions in this and the next two sections will be identical with the constructions in Section~\ref{sec-deformation}. \end{example}

\begin{example}  In the $2$-step case the Lie bracket condition in Definition~\ref{def-filtered-manifold} is vacuous, so   a 2-step filtered manifold is simply a smooth manifold together with a smooth  vector subbundle of the tangent bundle (Beals and Greiner \cite{BealsGreiner88} coined the term \emph{Heisenberg manifold}  for the special case in which this bundle has codimension one in the tangent bundle).  The calculations in this and the following sections are very easy in the $2$-step case.
\end{example}

For our purposes, the  significant features of  a filtered manifold $(V,H)$ will be accessed through the algebra of linear partial differential operators on $V$, and in particular through an increasing filtration on differential operators that is determined by the Lie filtration on $TV$.

We begin with some generalities on differential operators, unrelated to Lie filtrations. If  $X_1,\dots, X_n$ is any local frame for the tangent bundle of a smooth manifold, then any linear partial differential operator  $D$ can be expressed in a unique way as a linear combination 
\begin{equation}
\label{eq-diff-op-decomp}
 D =  \sum _\alpha f_\alpha X^\alpha ,
 \end{equation}
 where 
 \begin{enumerate}[\rm (i)]
 
 \item the sum is over all multi-indices $\alpha = (\alpha _1,\dots, \alpha _n) $ with nonnegative integer entries, 
 
 \item 
$
 X^\alpha =  X^{\alpha_1}_1 \cdots X_n^{\alpha_n}
$ (note that the order of $X_1,\dots, X_n$ is fixed), 
  and 
  
  \item   the coefficients  $f_\alpha$ are smooth functions, all but finitely many of them zero.
\end{enumerate}

\begin{lemma} 
\label{lemma-D-vanishing-at-v} 
Let $v$ be a point in a smooth manifold $V$, let $\{ X_1,\dots, X_n\}$ be a local frame for $TV$, defined near $v$.  If a linear differential operator $D$ is expressed in  terms of the frame as  in \eqref{eq-diff-op-decomp}, and if $D$   {vanishes at $v$} in the sense that $(Df)(v) = 0$ for every smooth function $f$ on $V$, then all the   functions $f_\alpha$ vanish at $v$. \qed
\end{lemma}

The following two  definitions are taken from the work of Choi and Ponge  {\cite[Section 2]{ChoiPonge15}} (which in turn adapts terminology from \cite[Section 4]{Bellaiche96}).

\begin{definition}
  Let $(V,H)$ be an $r$-step filtered manifold.  A \emph{local $H$-frame} for $V$ is a local frame $X_1,\dots, X_n$ for the tangent bundle such that for every $q = 1,\dots, r$,   the vector fields 
\[
X_1,
\dots, X_{\operatorname{rank}(H^q)}
\]
are sections of $H^q$, and so constitute a local frame for $H^q$.
\end{definition}

\begin{definition}
The \emph{weight sequence} of $V$ is the sequence
\[
(q_1, \dots , q_n) = (1,\dots, 1 ,2,\dots, 2,  \dots, r,\dots , r)
\]
in which each integer $q$ is repeated $\operatorname{rank}(H^q) - \operatorname{rank}(H^{q-1})$ times.
\end{definition}

\begin{remark}
With this terminology, if $\{ X_a\}$ is a local $H$-frame, then $X_a$ is a section of the vector bundle $H^{q_a}$.
 \end{remark} 
 
\begin{definition}[{\cite[Section 3]{Melin82}}]
Let $(V,H)$ be an $r$-step  filtered manifold. Let $D$ be a linear differential operator and let $s$ be a nonnegative integer.  We shall write 
\[
\order_H (D) \le s ,
\]
and say that the $H$-order of $D$ is no more than $s$, at a point $v\in V$, if  for some (or equivalently every)  local $H$-frame $X_1,\dots, X_n$ defined near  $v$, the operator $D$ can be expressed  as  a sum
\begin{equation}
 D = \sum _\alpha f_\alpha X^{\alpha_1}_1 \cdots X_n^{\alpha_n},
 \end{equation}
in such a way that 
\[
q_1\alpha_1 + \cdots + q_n \alpha_n > s \quad \Rightarrow \quad f_\alpha =0 ,
\]
where $\{ q_a\}$ is the weight sequence for $(V,H)$.
\end{definition}

\begin{example}
In the $1$-step  case (see Example~\ref{ex-ordinary-filtered})  this is of course the usual notion of order of a differential operator.  
\end{example}

\begin{definition} 
Let $(V,H)$ be a filtered manifold and denote by $\D (V)$ the algebra of linear partial differential operators on $V$.  We shall denote by 
\[
\D ^s (V) \subseteq \D (V)
\]
the linear space of all operators that are of $H$-order no more than $s$ at every point of $V$.
\end{definition}

It is evident that  if $p$ and $q$ are any nonnegative integers, then
\[
\D^p(V) \cdot \D ^q (V) \subseteq \D ^{p+q}(V),
\]
so the concept of $H$-order defines an increasing filtration on the algebra $\D(V)$.   If $X$ is a vector field on $V$, then  $X$ has $H$-order no more than   $q$ as a differential operator if and only if   it is a section of $H^{q}$.

The   notion of  $H$-order on differential operators  leads to the following    notion of order of vanishing of a function at a point in a filtered manifold:

\begin{definition} 
 Let $V$ be a filtered manifold and let $v$ be a point in $V$.  Let $q$ be a positive integer.
A smooth function $f$  on $V$ \emph{vanishes to $H$-order $q$  at $v$}  if  the function   $Df$ vanishes at $v$ for every differential operator   $D$ of $H$-order $q{-}1$ or less.  We shall  denote by 
\[
I_q (V,v) \subseteq C^\infty (V)
\]
the ideal of smooth, real-valued  functions on $V$ that vanish to $H$-order $q$.  For convenience we shall also  write  $I_0(V,v) = C^\infty (V)$.   
\end{definition}

Of course, even though the notation does not indicate it,   the ideals $I_q(V,v)$ depend on the filtration $H$. The spaces $I_q(V,v)$ decrease as $q$ increases, and in addition
\[
I_p(V,v) \cdot I_q(V,v) \subseteq I_{p+q} (V,v)
\]
for all $p , q \ge 0$.   
So we obtain a decreasing filtration of the algebra $ C^\infty (V)$ by ideals.

\begin{definition}
\label{def-A-zero-v}
Let $v$ be a point in a filtered manifold $(V,H)$.  Denote by $A_0(V,v)$ the associated graded algebra
\[
A_0 (V,v) =  \bigoplus _{q\ge 0} I_q(V,v)/I_{q{+}1}(V,v).
\]
\end{definition}

In the context of ordinary manifolds this naturally identifies with  the algebra of polynomial functions on the tangent space $T_vV$.  Our objective in the remainder of this section is to show that  $A_0(V,v)$  naturally identifies with the algebra of polynomial functions on a real unipotent group $\H_v$ attached to the Lie filtration and the point $v\in V$.

\begin{definition} 
\label{def-lie-algebra}
Let  $(V,H)$  be a filtered manifold and let $v\in V$.   Denote by $\mathfrak{h}_v$  the direct sum
\[
 {\mathfrak h}_v = \bigoplus_{q=1}^r  H_v^{q}/ H_v^{q-1} .
 \]
 Equip $\mathfrak{h}_v$ with a graded Lie algebra structure, as follows.  Given elements $\langle X_v\rangle_p$ and $\langle Y_v\rangle_q$ in degrees  $p$ and $q$, represented by tangent vectors  $X_v\in H_v^p$ and $Y_v\in H_v^q$, extend both  to sections of $H^p$ and $H^q$ and define
 \[
 \bigl [ \langle X_v\rangle _p , \langle Y_v\rangle _q \bigr ]  =   \langle  [X,Y]_v   \rangle_{p+q}.
 \]
For further details, and examples, see \cite{Melin82}, \cite{ChoiPonge15} or \cite{vanErpYuncken15}.
\end{definition}

\begin{lemma}
\label{lemma-rank-r-derivation}
The graded Lie algebra $\mathfrak{h}_v$ acts as   derivations on the graded algebra $A_0(V,v)$  via the   formula
\begin{equation*}
\label{eq-rank-r-derivation}
\delta_{\langle X_v\rangle _p} \colon \sum _{q \ge 0}  \langle a_q\rangle _q \longmapsto  \sum_{q \ge p}\langle X(a_q)\rangle _{q-p} ,
\end{equation*}
where $X_v$ is extended  to a     section $X$ of $H^p$, as in Definition~\textup{\ref{def-lie-algebra}} \textup{(}and where the angle-bracket notation $\langle a\rangle_q$ is as in \eqref{eq-assoc-graded-notation}\textup{)}. \qed
\end{lemma}

\begin{definition}
\label{def-osculating-group}
We shall denote by  $\H_v$   the unipotent group with Lie algebra $\mathfrak{h}_v$.  This is the \emph{osculating group} attached to the point $v$.  Denote by   $A(\H_v)$  the algebra of real-valued polynomial  functions  on $\H_v$.
\end{definition}

\begin{remark}
In the present context, \emph{unipotent group} means the same thing as \emph{simply connected nilpotent Lie group}, while  $A(\H_v)$     is the algebra of  functions on the group that  correspond to polynomial functions on the Lie algebra $\mathfrak{h}_v$ under the exponential map
\[
\exp \colon \mathfrak {h}_v   \longrightarrow \H_v ,
\]
which, we recall, is a diffeomorphism.   See for example \cite[Chapter XVI, Section 4]{Hochschild81} for a more algebraic   construction of $\H_v$.
 \end{remark}

Now if $A$ is an algebra  that is equipped with a locally finite-dimensional and locally nilpotent action of a finite-dimensional real nilpotent Lie algebra $\mathfrak{h}$ by derivations, then   the action of $\mathfrak{h}$ exponentiates to an action of the associated unipotent group $\H$ by algebra automorphisms.  And if  $\varepsilon$ is any character  of $A$, then there is an \emph{orbit homomorphism}\footnote{It is dual to the orbit map $\H\to \operatorname{Spectrum}(A)$ given by $h\mapsto h(\varepsilon)$.}
\begin{equation}
\label{eq-action-map}
A    \longrightarrow  A(\H)
\end{equation}
into the algebra of real-valued polynomial functions on the associated uni\-potent group that is defined by  the formula 
\begin{equation}
\label{eq-action-map-fmla}
{}\qquad a \longmapsto  \bigl [ h \mapsto \varepsilon(h^{-1}(a)) \bigr ] \qquad (a\in A , \quad h \in  \H).
\end{equation}
It  is an $\H$-equivariant algebra homomorphism if we let  $\H$   act  on $A(\H) $ by the left regular representation.

\begin{definition}
We shall call  the character 
\begin{equation*}
A_0(V,v) \ni \sum \langle a_q\rangle _q \stackrel \varepsilon \longmapsto a_0(v) \in \R
\end{equation*}
  the  \emph{counit}  of  $A_0(V,v)$. 
  \end{definition}
  
  We shall prove the following result.

\begin{theorem}
\label{thm-filtrations-and-unipotent-groups}
Let  $(V,H)$ be a filtered manifold, and let $v$ be a point in  $V$. The orbit homomorphism
\[
A_0 (V,v )  \longrightarrow   A(\H_v)
\]
 associated to the counit  of $A_0(V,v)$ is an $\H_v$-equivariant algebra isomorphism.
\end{theorem}

\begin{remark} 
The orbit homomorphism in the theorem is the \emph{unique} $\H_v$-equivariant homomorphism    for which the composition
\[
\xymatrix@C=40pt{
A_0(V,v) \ar[r] &    A(\H_v) \ar[r]^-{\text{eval. at $e$}} &  \R
}
\]
is the counit of    $A_0(V,v)$. 
\end{remark}
 
\begin{lemma}
\label{lemma-local-coords1}
Let $V$ be a filtered manifold of rank $r$, and  let $v$ be a point in $V$. Let $ \{ q_1,\dots, q_n\}$ be the weight sequence for $(V,H)$ and   let $\{ X_a \}$ be a local $H$-frame, defined near $v$.
There are local coordinates 
$
\{x_a\} 
$
defined near $v$ 
 such that 
\begin{enumerate}[\rm (i)]
\item 
each $x_a$ vanishes  at $v$ to $H$-order $q_a$, and 
\item $X_a(x_b) = \delta _{ab}$ at the point  $v$, for all $a,b=1,\dots , n$.
\end{enumerate}
\end{lemma}

\begin{proof}
Define a linear transformation  from  $\D(V)$ into the vector space dual of $C^\infty (V)$   by   the formula
\[
D \longmapsto \bigl [ f \mapsto (Df)(v) \bigr ] .
\]
It induces a  linear map 
\begin{equation}
\label{eq-fd-image}
\mathcal D^r(V)  \longrightarrow \bigl ( C^\infty (V) \big / I_{r+1}(V,v) \bigr ) ^* .
\end{equation}
Note that the  quotient    $C^\infty (V) \big / I_{r+1}(V,v)$ is a  \emph{finite-dimen\-sional} vector space.   
 
  It follows from Lemma~\ref{lemma-D-vanishing-at-v} that the images under  \eqref{eq-fd-image} of  the  monomial differential operators $X^\alpha$ of $H$-order no more than $r$   are linearly independent.  So by linear algebra there  are  functions $f_\beta\in C^\infty (V)$ with 
\[
(X^\alpha  f_\beta)(v) = \delta_{\alpha \beta}
\]
 The members  $\{x_a\} $  of this list of functions that correspond to the vector fields $\{ X_a\}$ form a local coordinate system of the required type.
 \end{proof}

\begin{remark}
\label{rem-privileged}
The coordinates provided by the  lemma above are called  \emph{privileged coordinates} in \cite[Definition 4.9]{ChoiPonge15} and \cite{Bellaiche96}, and their existence is   proved  in \cite[Proposition 4.13]{ChoiPonge15} and in \cite[Theorem 4.15]{Bellaiche96}.  Our argument is only slightly different.
 \end{remark}

 \begin{proof}[Proof of Theorem~\ref{thm-filtrations-and-unipotent-groups}]
Equip the algebra $A(\H_v)$  with the decreasing filtration given by order of vanishing, in the ordinary sense unrelated to Lie filtrations, at $e\in \H_v$. The associated graded algebra is the symmetric algebra on its degree one part, which identifies with  $\mathfrak{h}^*_v$. 

The algebra $A_0(V,v)$ also carries a  decreasing filtration, in which an element has order $j$ or more if it can be represented as a sum $\sum \langle a_q\rangle _q$, with each $a_q$ vanishing, also in the ordinary sense,  to   order $j$  or more.
The associated graded algebra is a symmetric algebra on the degree-one classes determined by the elements $\langle x_a \rangle_{q_a}$, where $\{x_a\}$ is any coordinate system as in  Lemma~\ref{lemma-local-coords1}.  

The filtrations of $A_ 0(V,v)$ and  $A(\H_v)$  are compatible with one another under the map \eqref{eq-action-map}, and the generators $\langle x_a \rangle_{q_a}$ map to the dual basis elements  
\[
 \langle X_{a,v}\rangle_{q_a}^* \in \mathfrak h^*_v,
 \]
  with $\{ X_a\} $ the local $H$-frame  in Lemma~\ref{lemma-local-coords1}.  This proves the theorem.\end{proof}

\begin{remark}
\label{remark-carnot1}
Let  $\{ X_a\}$ be a local $H$-frame near $v\in V$, and let $\{x_a\}$ be an associated system of privileged coordinates, as in Lemma~\ref{lemma-local-coords1}.  The frame determines a basis $\{\langle X_{a,v}\rangle_{q_a} \}$ for the Lie algebra $\mathfrak h_v$ and the local coordinates determine a local diffeomorphism 
\[
w\longmapsto \sum_a x_a(w) \langle X_{a,v}\rangle _{q_a}
\]
from $V$ to $\mathfrak h_v$, and hence, by exponentiation, a local diffeomorphism
\[
 V \stackrel \cong \longrightarrow \H_v.
\] 
This in turn  induces an isomorphism of algebras
\[
A (\H_v) \stackrel \cong \longrightarrow A_0(V,v).
\]
The algebra isomorphism  depends on the choice of coordinate systems $\{x_a\}$, in general, and is \emph{not} in general inverse to the canonical  isomorphism of Theorem~\ref{thm-filtrations-and-unipotent-groups}.   Those coordinates for which the two isomorphisms \emph{are} inverse to one another are called \emph{Carnot coordinates} in \cite{ChoiPonge15}.
\end{remark}


\section{Normal  Spaces for  Filtered Manifolds}
\label{sec-normal-spaces}

In this section we shall  construct  the filtered manifold analogue of the normal bundle.  Its fibers will be most naturally viewed as  unipotent homogeneous spaces rather than as quotients of tangent vector spaces.
 
\begin{definition}
\label{def-filtered-submanifold}
Let  $(V,H)$ be an $r$-step filtered manifold.  An embedded   submanifold $M\subseteq V$ is a  \emph{filtered submanifold} if the intersections 
\[
G ^q = H^q\vert_M \cap TM \qquad (q=1,\dots, r)
\]
are smooth vector subbundles of $TM$. 
\end{definition}

If $M$ is a filtered submanifold of $(V,H)$, then the  bundles $G^q$  form a Lie filtration of $TM$, so that $(M,G)$ is a filtered manifold in its own right.

\begin{definition}  Let $(M,G)$ be a  filtered submanifold   of a filtered manifold $(V,H)$, and denote by  $I_q(V,M)$   the ideal of smooth functions on $V$ that vanish to $H$-order at least $q$ on $M$. 
We shall denote by $A_0(V,M)$  the associated graded algebra 
 \[
 A_0(V,M) = \bigoplus_{q\ge 0} I_q(V,M) / I_{q+1}(V,M) 
 \]
 The  \emph{normal space} $N^H_VM$ is the spectrum of $A_0(V,M)$.
\end{definition}

\begin{theorem}
\label{thm-filtered-A-zero}
Let $(M,G)$ be a filtered submanifold of a filtered manifold $(V,H)$.
The normal space $N^H_VM$ is a smooth manifold in such a way that  the sheaf of smooth functions  is the sheaf   from Definition~\textup{\ref{def-sheaf}}.\end{theorem}

The proof is not difficult, but it requires some information about vector fields and local coordinates adapted to the inclusion of $M$ into $V$.

\begin{definition}
Let  $(M,G)$ be a filtered submanifold of a filtered manifold $(V,H)$. 
A \emph{local $(G,H)$-frame} for $TV$ at a point of $M$  is a local $H$-frame for $V$  with the additional property that the vector fields in the frame that are tangent to $M$ (upon restriction to $M$) form a local $G$-frame for $M$.
\end{definition}

The vector fields in the local frame divide into two sets: 
   \begin{enumerate}[\rm (i)]
   
   \item vector fields tangent to $M$ upon restriction to $M$, which restrict to give a $G$-local frame for $M$, and 
   
   \item vector fields  not tangent  to $M$.
 \end{enumerate}
  We shall call the latter  the \emph{normal} vector fields in the local frame. The normal vector fields $X_a$  for which  $a \le \operatorname{rank}(H^p)$ restrict to give a  local frame  for  the quotient bundle
$
H^p\big \vert _M \, \big / \, G^p .
$

\begin{lemma}
\label{lemma-local-coords2}
Let $(V,H)$ be an $r$-step filtered manifold  with order sequence $\{ q_a\}$, and let $(M,G)$ be a filtered submanifold of $V$. Let $\{ X_a\}$ be a local $(G,H)$-frame defined near a point  $m\in M$.  There are  smooth functions $z_c$ defined near $m$, one for each normal vector field $X_c$ in the frame,  such that 
\begin{enumerate}[\rm (i)]

\item 
$z_c$ vanishes on $M$ to $H$-order $q_c$.

\item $X_c(z_d) = \delta _{cd}$ on $M$.

\end{enumerate}
\end{lemma}

To prove this generalization of Lemma~\ref{lemma-local-coords1} we shall   use the following generalization of Lemma~\ref{lemma-D-vanishing-at-v}.

\begin{lemma} 
\label{lemma-D-vanishing-at-v2} 
Let $M$ be an embedded submanifold of a smooth manifold $V$, and let $m$ be a point in  $M$. Let $\{ Z_1,\dots, Z_k\}$ be vector fields on $V$, defined in some neighborhood of $m\in V$, and assume that their values at $m$ project to linearly independent vectors in the normal space $TV\vert _M / TM$.  If a linear differential operator of the form 
\[
D = \sum f_\alpha Z^\alpha
\]
has the property that $(Df)(m) = 0$ for every smooth function $f$ on $V$ that vanishes on $M$, then all the coefficient functions $f_\alpha$ vanish at $m$. \qed
\end{lemma}

\begin{proof}[Proof of Lemma~\ref{lemma-local-coords2}]
Acccording to Lemma~\ref{lemma-D-vanishing-at-v2} the monomial operators  $X^\alpha$ that use only normal vector fields in the local $(G,H)$-frame map by evaluation at $m$ to a linearly independent set in $\operatorname{Hom} ( I_1(V,M), \R)$. If we consider only monomial operators of $H$-order $r$ or less, then this linearly independent set lies in the finite-dimensional vector space 
\[
\operatorname{Hom} ( I_1(V,M)/ I_{r+1}(V,M), \R)
\subseteq 
\operatorname{Hom} ( I_1(V,M), \R)
\]
and so, by linear algebra, associated to this finite linearly independent set in a finite-dimensional vector space there are functions $g_\beta \in I_1(V,M)$ with $X^\alpha(g_\beta) = \delta _\alpha ^\beta$ at the point $m$.  

We want  to adjust the functions $g_\beta$ so that this relation holds near $m$ in $M$, not only at the single point $m$.
Let $h_{\alpha\beta} = X^\alpha (g_\beta)$. This matrix of functions is the identity at $m$, and so is invertible near $m$.  Let $h^{\alpha\beta}$ be the entries of the inverse matrix and define 
\[
f_\beta = \sum _\gamma h^{\beta\gamma}g_\gamma.
\]
Then $X^\alpha(f_\beta) = \delta_{\alpha\beta}$ on $M$, near $m$.
Now, if we define $ z_c$ to be the function $f_\beta$ associated to the vector field $  X_c\in \{ X^\beta\}$, then the functions $\{ z_c\}$ have the required properties.
\end{proof}

\begin{proof}[Proof of Theorem~\ref{thm-filtered-A-zero}] 
We shall use the vector fields and functions obtained above to show that the criteria in Lemma~\ref{lemma-mfld-characterization} are satisfied for every character $\varphi$  of $A_0(V,M)$.  

The degree zero part of $A_0(V,M)$ is $C^\infty(M)$, and $\varphi$  restricts there to evaluation at some $m\in M$. 
Let $\{ X_a\}$ be a local $(G,H)$-frame near $m$.  Choose smooth functions $\{z_c\}$ on $V$ as in Lemma~\ref{lemma-local-coords2}. In addition, choose   smooth functions $\{y_a\}$ on $V$, indexed by the  members $Y_a$ of the local $(G,H)$-frame that are tangent to $M$, so that 
\[
Y_a(y_b) = \delta_{ab}\quad \text{at $m\in V$}.
\]
There is a neighborhood $U$ of $m\in V$ such that  functions $\{ y_a,z_c\}$  are   coordinates for  $U$, while the functions $\{y_a\}$ restrict to   coordinates for $M\cap U$.

Now let $\Lambda$ be the open set in $N^H_VM$ consisting of all those characters whose restriction to the degree zero part of $A_0(V,M)$ is evaluation at some point of $M \cap U$.  It follows from Taylor's theorem that the  elements 
\begin{equation}
\label{eq-smooth-generators-A-zero}
 \langle y_a\rangle _0 \quad \text{and } \quad  \langle z_c \rangle _{q_c}
\end{equation}
smoothly generate $A_0(V,M)$ over $\Lambda$.  

Moreover  $A_0(U,M\cap U)$ is freely generated as an algebra over its degree zero part $C^\infty (M\cap U)$ by the classes $\langle z_c\rangle _{q_c}$.  So if $\dim(M){=}k$ and $\dim (V) {=} n$, then the map 
\[
\operatorname{Spectrum}(A_0(V,M)) \longrightarrow \R^k \times \R ^{n-k}
\]
given by evaluation on the generators \eqref{eq-smooth-generators-A-zero} sends $\Lambda$ homeomorphically to the open set 
$W\times \R^{n-k}$, where $W\subseteq \R^k$ is the range of the coordinates $\{y_c\}$ on $M \cap U$.
\end{proof}

We shall now calculate the normal space $N^H_VM$ in terms of the osculating groups  introduced in the last section.
There is a natural map 
\begin{equation}
\label{eq-normal-space-submersion}
N^H_V M \longrightarrow M
\end{equation}
corresponding to the inclusion of $C^\infty(M)$ as the degree zero subalgebra of $A_0(V,M)$, and fiber of $N^H_VM$  over    $m\in M$ identifies with the  spectrum of  the following algebra.

\begin{definition}
If $m\in M$, then we shall denote by 
$
A_{0,m}(V,M) 
$
the quotient of $A_0(V,M)$ by the ideal in $A_0(V,M)$ generated by the vanishing ideal of $m$ in $C^\infty (M)$. 
The formula 
\[
\varepsilon_m\colon \sum \langle a_q\rangle_q \longmapsto a_0(m)
\]
defines a character of $A_{0,m}(V,M)$ that we shall call the \emph{counit}.
\end{definition}

\begin{theorem}
\label{thm-filtered-A-zero-fiber}
Let $(M,G)$ be a filtered submanifold of a filtered manifold $(V,H)$ and let $m $ be a point in $M$.  Let $\H_m$ and $\G_m$ be the osculating groups for $m\in V$ and $m\in M$, respectively.
There is a unique  $H_m$-equivariant algebra  isomorphism
\[
A_{0,m}(V,M) \longrightarrow A( \H_m /\G_m)
\]
whose composition with evaluation at the identity coset in $\H_m/\G_m$ is the       counit $\varepsilon_m$ of $A_{0,m}(V,M)$.
\end{theorem}

\begin{remark}
Here  $A( \H_m /\G_m)$ is the algebra of polynomial functions on the unipotent homogenous space $ \H_m /\G_m$, or equivalently the algebra of polynomial functions on $\H_v$ that are invariant under right translations by elements of $\G_m$.
\end{remark}

\begin{proof}
The Lie algebra $\mathfrak h_m$ acts on $A_{0,m}(V,M)$ by derivations according to the formula in Lemma~\ref{lemma-rank-r-derivation}, and this action exponentiates to a locally finite-dimensional  action of $H_v$ by automorphisms.  The image of the  orbit map
\[
A_{0,m}(V,M) \longrightarrow A(\H_m)
\]
associated to the counit $\varepsilon_m$ is included in the right $\G_m$-invariant functions on $A (\H_m)$; this is a consequence of the fact that if $X\in \mathfrak g_m$, then 
\[
\varepsilon_m(\delta_X(a)) = 0
\]
for every $a\in A_{0,m}(V,M)$.  So we obtain an orbit homomorphism
\[
A_{0,m}(V,M) \longrightarrow A( \H_m/\G_m) ,
\]
and it remains to show that it is an isomorphism.  We shall use  a variation on the argument used to prove Theorem~\ref{thm-filtrations-and-unipotent-groups}.  

Filter  $A_{0,m}(V,M) $ by order of vanishing of functions in the ordinary sense at $m$.  Using the coordinates of the previous lemma, the associated graded algebra is freely generated by the classes $\langle z_c\rangle _{q_c}$.  

Filter $A( \H_m/\G_m)$ by order of vanishing in the ordinary sense at the basepoint in $\H_m/\G_m$.  The associated graded algebra is freely generated by the normal dual vectors $\langle Z_c\rangle ^*\in (\mathfrak h_m/\mathfrak g_m)^*$. 

Our orbit map is filtration preserving,  we find that it induces an isomorphism on associated graded algebras; indeed it maps $\langle z_c \rangle _{q_c}$ to $\langle Z_c\rangle^*$.
\end{proof}

\begin{remark}
The algebra  $A_0(V,M)$ consists of those smooth functions on the normal space $N^H_VM$ whose restrictions to all of  the fibers of \eqref{eq-normal-space-submersion} are  polynomial functions.
\end{remark}

\section{Deformation Spaces for Filtered Manifolds}

In this section we shall construct the deformation space associated to a filtered submanifold of a filtered manifold.  We shall copy Section~\ref{sec-deformation}  almost verbatim.  

\begin{definition}  Let $(M,G)$ be a filtered submanifold of a filtered manifold $(V,H)$.
Denote by  $A(V,M)$ the algebra of  Laurent polynomials 
 \[
 \sum _{n\in \Z} a_q t^{-q}
 \]
whose coefficients are smooth, real-valued  functions on $V$ that satisfy the condition
\[
q > 0 \quad \Rightarrow \quad \text{$a_q$ vanishes to $H$-order $q$ on $M$} .
\]
 The \emph{deformation space} $\N^H_VM$ is the spectrum of $A(V,M)$.
\end{definition}

As is the case for ordinary manifolds, the deformation space is a union 
\begin{equation*}
\N^H_VM = N^H_VM {\times} \{ 0\}\, \sqcup\, \bigsqcup _{\lambda \in \R ^{\times}}  V {\times} \{ \lambda\} ,
\end{equation*}
(but of course with the normal space from the previous section).
 
\begin{theorem}
\label{thm-def-space-mfld2}
The deformation space $\N^H_VM$ is a smooth manifold  in such a way that  the sheaf of smooth functions  is the sheaf   from Definition~\textup{\ref{def-sheaf}}.
\end{theorem}

\begin{proof}
We shall follow the proof of Theorem~\ref{thm-def-space-mfld}, and we shall use the same coordinate functions $\{y_a\}$ and $\{z_c\}$  as in the proof of  Theorem~\ref{thm-filtered-A-zero}, defined in a neighborhood $U$ of $m\in V$.  Let 
$\Lambda \subseteq \N_VM$ be the open subset consisting of all $(u,\lambda)$ with $u\in U$ and $\lambda \ne 0$, together with all the elements $(X_m,0)$, with $X_m\in \H_m/G_m$.  The elements
\begin{equation}
\label{eq-def-space-coords2}
t, \quad y_a, \quad\text{and} \quad z_ct^{-q_c}
\end{equation}
of $A(V,M)$ satisfy the conditions of Lemma~\ref{lemma-mfld-characterization}.  If $W\subseteq \R^n$ is the image of the coordinates $\{ y_a,z_c\}$, then the functions \eqref{eq-def-space-coords2} map $\Lambda$ homeomorphically to  the open subset 
\[
\left \{\, 
\bigl (\lambda, \{y_a\},\{z_c\} \bigr ) \, : \, \bigl  (\{y_a\},\{ \lambda^{q_{c}} z_{c}\} \bigr) \in W 
\,\right  \} 
\]
of $\R^{n+1}$.
\end{proof}

\section{The Tangent Groupoid for Filtered Manifolds}
\label{sec-tangent-groupoid-filtered}

In this section we shall briefly discuss the diagonal embedding of a filtered manifold into its square, where the deformation space carries a Lie groupoid structure. We shall describe this groupoid structure in terms of the osculating groups in Definition~\ref{def-osculating-group}.

 \begin{definition}
 \label{def-tangent-groupoid-filtered}
 Let $(M,G)$ be a filtered manifold, and define  a Lie filtration of $M{\times}M$ by defining $H^p \subseteq TM{\times}TM$  to be   $G^p{\times} G^p$.  The  \emph{tangent groupoid} of $(M,G)$   is  the deformation space 
\begin{equation*}
\T ^G M : = \N^H _{  M} \, M{\times}M.
\end{equation*}
associated to the diagonal embedding of $M$ in $M{\times}M$.
 \end{definition}
 
  The   tangent groupoid for filtered manifolds  was   previously constructed   by  Van Erp \cite{vanErp05} and  Ponge   \cite{Ponge06}  in the $2$-step case, and then by    Choi and Ponge  \cite{ChoiPonge15}, and also by Van Erp and Yuncken  \cite{vanErpYuncken16},    in the general case.   
  Connes gave a proof of  the Atiyah-Singer theorem using the standard tangent group\-oid considered in Section~\ref{sec-tangent-groupoid} \cite[Chapter 2, Section 5]{Connes94}. See  \cite{vanErp10a} for a proof of an index theorem for contact manifolds  using a similar approach. 
  
As in Section~\ref{sec-tangent-groupoid}, the tangent groupoid has a natural Lie groupoid structure with object space $M{\times}\R$.  
The part of  $\T M$ over  each $\lambda \ne 0$ is  a copy of the pair groupoid of $M$, as before, and it remains to describe the groupoid structure over $\lambda =0$.

If $\mathcal{G}_m$ is the osculating group at $m\in M$, as in Definition \ref{def-osculating-group}, then the isomorphism of  Theorem~\ref{thm-filtered-A-zero-fiber} gives an identification
\begin{equation}
\label{eq-zero-fiber-identifications}
\T^G M \vert _{(m,0)} \cong  \left ( \mathcal G_m{\times} \mathcal G_m \right ) / \mathcal G_m \cong \mathcal G_m .
\end{equation}
Here $\mathcal{G}_m$ is embedded diagonally as a subgroup of $\mathcal G_m{\times} \mathcal G_m $, and the second isomorphism is induced from $(g_1,g_2) \mapsto g_1g_2{}^{-1}$.
\begin{proposition}
\label{prop-tangent-group-structure}
The multiplication on the fiber of $\T^G M$ over $(m,0)$ that is induced from the groupoid structure on $\T^G M$ is the same as the group multiplication operation that is  induced from the identification \eqref{eq-zero-fiber-identifications}.
\end{proposition}

To prove the proposition, let us return to the functoriality of the deformation space that was mentioned (for ordinary manifolds) in Section~\ref{sec-tangent-groupoid}.   Suppose given a commutative diagram
\[
\xymatrix@C=40pt{
M_1 \ar[r] \ar[d] & M_2 \ar[d]\\
V_1 \ar[r]_\varphi & V_2
}
\] 
in which the columns are inclusions of filtered manifolds, as in Definition~\ref{def-filtered-submanifold}, and the differentials of the horizontal maps are filtration-preserving on tangent spaces.  There is an induced map on deformation spaces, and in particular on  normal spaces.  Indeed if $\varphi(m_1) = m_2$ then composition with $\varphi$  induces a morphism of algebras
\begin{equation}
\label{eq-induced-map-on-normal-space}
\varphi^* \colon A_{0,m_2} (V_2,M_2) \longrightarrow A_{0,m_1}(V_1,M_1) .
\end{equation}
In addition, the differential of $\varphi$ induces a Lie algebra homomorphism 
\begin{equation}
\label{eq-induced-map-on-lie-algebra}
\varphi_* \colon \mathfrak{h}_{1,m_1} \longrightarrow \mathfrak{h}_{2,m_2}
\end{equation}
and so a group morphism 
\begin{equation}
\label{eq-induced-map-on-lie-group}
\varphi_* \colon \mathcal{H}_{1,m_1} \longrightarrow \mathcal{H}_{2,m_2}.
\end{equation}
The morphisms \eqref{eq-induced-map-on-normal-space} and \eqref{eq-induced-map-on-lie-algebra} are related as follows: if $f\in A_{0,m_2}(V_2,M_2)$, then
\begin{equation}
\label{eq-equivariance-formula}
 \delta_{\xi_1} \varphi^*f = \varphi^*  \delta_{\varphi_*\xi_1} f
 \qquad \forall \xi_1 \in \mathfrak{h}_{m_1}
\end{equation}
(for ordinary manifolds  this is simply   the definition of the differential $\varphi_*$). 

Consider now the induced map on normal spaces 
\[
\varphi _* \colon N^{H_1}_{V_1} M_1\big \vert _{m_1}  \longrightarrow  N^{H_2}_{V_2} M_2\big \vert _{m_2} 
\]
(recall that the normal spaces are the spectra of the  algebras in \eqref{eq-induced-map-on-normal-space}).  Identify the normal spaces with unipotent homogeneous spaces, as in Theorem~\ref{thm-filtered-A-zero-fiber}, to obtain a map
\begin{equation}
\label{eq-equivariance-formula2}
\varphi_* \colon \mathcal{H}_{1,m_1} / \mathcal {G}_{1,m_1} \longrightarrow \mathcal{H}_{2, m_2} / \mathcal {G}_{2,m_2} .
\end{equation}
We find from \eqref{eq-equivariance-formula} that \eqref{eq-equivariance-formula2} is induced from \eqref{eq-induced-map-on-lie-group}.

\begin{proof}[Proof of Proposition~\ref{prop-tangent-group-structure}]
It follows from \eqref{eq-equivariance-formula} that the groupoid operation
\[
\T^G M \vert _{(m,0)}  \times \T^G M \vert _{(m,0)}  \longrightarrow \T^G M \vert _{(m,0)} ,
\]
when viewed as a map 
\[
 \mathcal G_m  \times  \mathcal G_m  \longrightarrow  \mathcal G_m 
 \]
 using \eqref{eq-zero-fiber-identifications}, is equivariant for the left and right multiplication actions of $\mathcal G_m$ (on the left and right factors, respectively, in the case of the left-hand side).  In addition, the groupoid operation maps $(e,e)$ to $e$. So it must be group multiplication.
\end{proof}

 \section{Euler-Like Vector Fields  on  Filtered Manifolds}

\begin{definition}
\label{def-euler-like-filtered}
Let $(M,G)$ be a filtered submanifold of a filtered manifold $(V,H)$.  An \emph{Euler-like vector field} for the embedding of $M$ into $V$ is a vector field $E$ with the property that if $f$ is a smooth function on $V$ that vanishes on $M$  to $H$-order $q$, then 
\[
E(f) = q \cdot f + r
\]
where $r$ is a smooth function that vanishes on $M$  to $H$-order $q{+}1$ or higher.
\end{definition}

\begin{example}
If $m\in M$ and if $\{y_a,z_c\}$ is the local coordinate system  defined near $m\in V$, that was used in the proofs of Theorems \ref{thm-filtered-A-zero} and \ref{thm-def-space-mfld2}, then formula 
\[
E = \sum_c q_c\cdot  z_c \cdot \frac{\partial}{\partial z_c}
\]
defines an Euler-like vector field  near $m$.  A global Euler-like vector field can be assembled from locally defined Euler-like vector fields of this type using a partition of unity.
\end{example}

Our aim is to relate Euler-like vector fields to tubular neighborhood embeddings, as in Theorem~\ref{thm-blm}.  An interesting feature of the filtered manifold case that we are now considering is that it is not immediately clear what the appropriate notion of tubular neighborhood embedding should be (for instance, the normal space $N^H_VM$ is not itself a filtered manifold, so we cannot insist that tubular neighborhood embeddings be isomorphisms of filtered manifolds).  So we shall let the analogue of Theorem~\ref{thm-blm} determine the definition of a tubular neighborhood embedding.

To define the appropriate notion of a tubular neigborhood embedding we shall need to define a ``zero section'' of the normal space, and then examine the vertical tangent bundle for the submersion 
\[
 N^H_VM \longrightarrow M
 \]
 at the zero section.  First, the homomorphism 
 \[
 \begin{gathered} 
 A_0(V,M) \longrightarrow C^\infty (M) \\
 \sum \langle a_q\rangle _q \longmapsto \langle a _0\rangle _0
 \end{gathered}
 \]
 defines an inclusion of $M$ into $N^H_VM$ that will be our zero section.  Next, the  vertical tangent space at a point $m$ in the zero section identifies with the quotient of Lie algebras $\mathfrak h_m / \mathfrak g_m$.   Each of $\mathfrak h_m$ and $\mathfrak g_m$ is a graded Lie algebra, and we shall write
 \[
 \mathfrak h^q_m =  H^q_m/H^{q-1}_m
 \quad
 \text{and}
 \quad 
  \mathfrak g^q_m =   G^q_m/G^{q-1}_m .
 \]

\begin{definition}
\label{def-tubular-nbd-embedding}
Let $(M,G)$ be a filtered submanifold of a filtered manifold $(V,H)$.  A \emph{tubular neighborhood embedding} of $N^H_VM$ into $V$ is a diffeomorphism from a neighborhood of $M\subseteq N^H_VM$ to a neighborhood of $M \subseteq V$ with the following properties:
\begin{enumerate}[\rm (a)]
\item The diffeomorphism is the identity on $M$
\item At each point of $M$ the differential maps the vertical space $\mathfrak h^q_m/\mathfrak g^q_m$ into $H^q_m$, and the composition 
\[
\mathfrak h^q_m/\mathfrak g^q_m\longrightarrow H^q_m \longrightarrow \mathfrak h^q_m/\mathfrak g^q_m 
\]
with the natural projection is the identity.
\end{enumerate}
\end{definition}

The normal space $N^H_VM$  carries a natural vector field, which we shall call the Euler vector field, as follows:

\begin{definition}
The  \emph{Euler vector field} on  $N^H_VM$ is the vector field associated to the smooth derivation of $A_0(V,M)$ given by 
\[
\sum _q \langle a _q\rangle _q \longmapsto \sum _q q\cdot \langle a _q\rangle _q.
\]
\label{def-euler-graded-bundle}
\end{definition}

 \begin{remark}
 The normal space $N^H_VM$ is not naturally a filtered manifold, in general.  But if $M$ is a point, then $N^H_VM$ is simply the unipotent group $\H_v$ and this is a filtered manifold.  In this case, the Euler vector field is Euler-like in the sense of Definition~\ref{def-euler-like-filtered}.
 \end{remark}

The Euler vector field generates a flow $\{ \rho_s\}$ on $N^H_VM$ that is easy to describe in group-theoretic terms.  First, there is a  one-parameter group of Lie algebra automorphisms of the graded  Lie algebra 
\[
\mathfrak h_m = \bigoplus _{q=1}^r H^q_m/H^{q-1}_m
\] that multiplies the degree $q$ summand by $e^{tq}$.  This one-parameter group exponentiates to a one-parameter group of automorphisms of the unipotent group $\H_m$ that maps the subgroup $\G_m$ to itself, and therefore induces a flow $\{ \rho _s\}$ on the homogeneous space $\H_m / \G _m$, as required.

\begin{definition}
Denote by $\boldsymbol{C}$ the vector field on $\N^H_VM$ that generates the flow 
\[
\gamma _s \colon \left \{ 
		\begin{aligned}  (v,\lambda) & \longmapsto (v, e^s\lambda) \\
					(X,0) & \longmapsto (\rho_{-s}X,0)\end{aligned}\right .
\]
\end{definition}

\begin{lemma}
\label{lemma-T-field2}
If $E$ is an  Euler-like vector field for the inclusion of $M$ into $V$, then the vector field
\[
\boldsymbol{T} =  {\lambda}^{-1}   E + \frac{\partial} {\partial \lambda}
\]
on the open subset $V {\times} \R^{\times}\subseteq \N^H_V M$ extends to a 
vector field on $\N^H_VM$ with 
\[
 \lambda \cdot \boldsymbol{T} = \boldsymbol{C}+ \boldsymbol{E} ,
\]
where $\boldsymbol{E}$ smoothly extends the $\lambda$-independent vector field on $V{\times}\R^{\times}$ that is defined by $E$. \qed
\end{lemma}

Repeating the argument from Section~\ref{sec-vector-fields-deformation-space} we find that: 
 
\begin{theorem}
\label{thm-filtered-embeddings}
Let $(M,G)$ be a filtered submanifold of a filtered manifold $(V,H)$.  
The correspondence that associates to each tubular neighborhood embedding the associated Euler-like vector field on $V$ is bijection from germs of tubular neighborhood embeddings to germs of Euler-like vector fields. \qed
\end{theorem}
 
\begin{remark}
\label{remark-carnot2}
In the case where $M$ is a point, the inverse 
\[
V \longrightarrow \H_m
\]
of the tubular neighborhood embedding corresponds to a system of Carnot coordinates, as in \cite[Section 7]{ChoiPonge15} and Remark~\ref{remark-carnot1}.
\end{remark}

\bibliography{Refs}
\bibliographystyle{alpha}

\end{document}